\documentclass[11pt, a4paper]{article}
\usepackage{graphicx}
\usepackage{xcolor}
\usepackage{authblk}

\usepackage{listings}
\usepackage{graphicx}
\usepackage{amssymb, amsmath, amsfonts}
\usepackage{amsthm}
\usepackage{subcaption}

\usepackage{booktabs}
\usepackage{siunitx}
\usepackage{arydshln, mathtools}
\usepackage[mathscr]{euscript}
\usepackage{cleveref}
\usepackage{bm}         
\usepackage[title]{appendix}
\usepackage{enumitem}

\usepackage{vmargin}
\setmarginsrb{3 cm}{0.5 cm}{3 cm}{1.5 cm}{1 cm}{1.5 cm}{1 cm}{1.5 cm}
\usepackage[backend=biber,style=numeric,
            sorting=none]{biblatex} 
\renewcommand{\eqref}[1]{(\ref{#1})}

\newcommand{\reals}[1][\empty]{\mathbb{R}^{#1}}

\newcommand{\linspan}{\mathrm{span}}

\newcommand{\norm}[1]{\Vert #1 \Vert}

\newcommand{\rank}{\mathrm{rank}}

\newcommand{\kron}{\otimes}

\newtheorem{theorem}{Theorem}
\newtheorem{lemma}{Lemma}

\begin{document}
\title{A New Upper Bound for the $d$-dimensional Algebraic Connectivity of Arbitrary Graphs}
\date{}
\author[1]{Juan F. Presenza\thanks{Email: jpresenza@fi.uba.ar. Juan F. Presenza was partially supported by the Universidad de Buenos Aires PhD scolarhsip.
}}
\author[2, 3]{Ignacio Mas}
\author[2, 3]{Juan I. Giribet}
\author[1, 2]{J. Ignacio Alvarez-Hamelin}

\affil[1]{\small Universidad de Buenos Aires, Facultad de Ingenier\'ia, Argentina}
\affil[2]{Consejo Nacional de Investigaciones Cient\'ificas y T\'ecnicas (CONICET)\\Argentina}
\affil[3]{Universidad de San Andr\'es, Argentina}

\maketitle

\begin{abstract}
    In this paper we show that the $d$-dimensional algebraic connectivity of an arbitrary graph $G$ is bounded above by its $1$-dimensional algebraic connectivity, i.e.,
    \begin{equation*}
        a_d(G) \leq a_1(G),
    \end{equation*}
    where $a_1(G)$ corresponds the well-studied second smallest eigenvalue of the graph Laplacian.
\end{abstract}

\section{Introduction}
\label{sec:introduction}
Let $G = (V, E)$ be an abstract graph given by a finite set of vertices $V$ and a set of edges $E$ (unordered pairs of distinct vertices). A framework is a realization of $G$ in $d$-space, given by the pair $(G, p)$, where $p = (p_i)_{i \in V}$ is an assignment of positions to each vertex in $\reals[d]$.
A framework is called rigid if every continuous trajectory of the vertices that passes through $p$ and preserves all edge lengths in $E$, also preserves the distances between every pair of vertices. In other words, if every continuous motion of $(G, p)$ sufficiently close to $p$ is a rigid body motion.
The stronger property of infinitesimal rigidity is defined as follows.
A framework $(G, p)$ is said to be infinitesimally rigid if, for every infinitesimal motion $u = (u_i)_{i \in V}$ that instantaneously preserves edge lengths, i.e., 
\begin{equation}
    (p_i - p_j)^T (u_i - u_j) = 0 \quad \forall ij \in E,
    \label{eq:edge_constraints}
\end{equation}
there is a smooth rigid motion of $(G, p)$ such that $u$ is the instantaneous velocity when passing through $p$. 
It is well known that both rigidity and infinitesimal rigidity are generic properties of a graph, which means that $G$ es either rigid (resp. infinitesimally rigid) or flexible (resp. infinitesimally flexible) in a dense open subset of generic configurations of $G$, which are the positions $p \in \reals[d|V|]$ that maximize the rank of the rigidity matrix.
It is also known that for generic configurations, rigidity and infinitesimal rigidity coincide.

The rigidity matrix, which account for the set of linear constraints in \eqref{eq:edge_constraints}, is defined as follows.
Let $m \coloneqq \dim(\{p_i\}_{i \in V})$ be the dimension of the affine hull of position set, and let $\delta_{ij} \in \reals[d]$ be an \textit{edge-direction} vector defined for the pair $ij \in E$ as follows
\begin{equation}
    \delta_{ij} \coloneqq 
    \begin{cases}
        \frac{p_i - p_j}{\norm{p_i - p_j}} \quad & \text{if } p_i \neq p_j, \\
        x \quad & \text{if } p_i = p_j \text{ and } m = 1, \\
        0 \quad & \text{otherwise},
    \end{cases}
    \label{eq:edge_direction}
\end{equation}
where the unit vector $x \in \reals[d]$ is such that $p_i - p_j \in \linspan\{x\}$ for all $i, j \in V$.
Unlike earlier work \cite{Jordan2020, Lew2022}, we let $\delta_{ij}$ be nonzero for a pair of adjacent vertices sharing the position, given that the position set is contained in a one dimensional affine space.
This subtle change is motivated by the intention of enhancing the relationship between the rigidity matrix with the graph Laplacian, and will be clarified in the rest of this paper. Note that, in the case that $m \neq 1$, our definition is equivalent to that presented in the aforementioned literature.
Now, given an ordering of the elements in $E$, the \textit{normalized rigidity matrix} $R_G(p)$ is defined as the $|E| \times d|V|$ matrix containing a row for each edge in the graph, and a block of $d$ columns for each vertex, such that the $ij$-row is
\begin{equation*}
    r_{ij} = (0, \ldots, 0, \delta_{ij}^T, 0, \ldots, 0, -\delta_{ij}^T, 0, \ldots, 0),
\end{equation*}
where the positions of the vectors $\delta_{ij}$ and $-\delta_{ij}$ correspond to the blocks associated to $i$ and $j$ respectively.

The infinitesimal rigidity of a framework is usually stated in terms of the rank of the rigidity matrix, see \cite{Asimow1979} for details. 
The framework $(G, p)$ is said to be infinitesimally rigid if $\rank (R_G(p)) = d|V| - D$ where $D \coloneqq (m + 1)(2d - m)/2$ is the dimension of the manifold of configurations in $\reals[d|V|]$ that are equal to $p$ up to a rigid motion of the framework.

Since $\rank (R_G(p))$ is always less or equal than $d|V| - D$, it follows that if $(G, p)$ is infinitesimally rigid, then $p$ must be a generic configuration of $G$, therefore $G$ is generically rigid in $d$-space.
Alternatively, infinitesimal rigidity can be expressed in terms of the positive semi-definite \textit{stiffness matrix} $L_G(p) \coloneqq R_G(p)^T R_G(p)$, whose kernel equals that of $R_G(p)$, which has dimension greater or equal than $D$. Therefore $(G, p)$ is infinitesimally rigid if and only if $\lambda_{D+1}(L_G(p)) > 0$, where $\lambda_{D+1}$ denotes the $(D+1)$-th smallest eigenvalue of $L_G(p)$, called the \textit{rigidity eigenvalue}. Now, let $T(p) \subseteq \ker(R_G(p))$ be the $D$-dimensional subspace of velocities associated with rigid motions, called \textit{trivial subspace}. Every element $u_T \in T(p)$ can be written as
\begin{equation*}
    u_T = (t + A p_i)_{i \in V}
\end{equation*}
where $t \in \reals[d]$ represents a translational velocity and $A p_i$ a rotational velocity, with $A$ being a $d \times d$ skew-symmetric matrix.

\section{Generalized Algebraic Connectivity}
It has been noticed in previous work that the stiffness matrix of a framework $(G, p)$ can be considered as a geometric generalization of the graph Laplacian $\mathcal{L}_G$.
In fact, for any realization $(G, p)$ on the real line ($m = d = 1$), it follows that $L_G(p)$ equals $\mathcal{L}_G$. This is also true when $p_i = p_j$ for some pair $ij \in E$, due to \eqref{eq:edge_direction}.
This connection allows use the vast background in spectral graph theory to enhance the rigidity theory, seen as a geometric generalization to higher dimensions.
It is well known that the algebraic connectivity of a graph is given by second smallest eigenvalue of $\mathcal{L}_G$.
This motivated Jord\'an and Tanigawa \cite{Jordan2020} to extend the analysis for $d > 1$ and define the $d$\textit{-dimensional algebraic connectivity} of a graph as the non-negative real number
\begin{equation}
    a_d(G) \coloneqq \sup \{\lambda_{D+1}(L_G(p)) : p \in \reals[d|V|]\}.
    \label{eq:d_algebraic_connectivity}
\end{equation}
From this definition, it holds that $a_d(G) > 0$ if and only if $G$ is generically rigid in $d$-space.
It is straightforward to see that $a_1(G) = \lambda_2(\mathcal{L}_G)$.

Preliminary results regarding lower and upper bounds for the $d$-dimensional algebraic connectivity where presented along with its definition in \cite{Jordan2020}.
Recently, Lew et al.~\cite{Lew2022} extended these results with new bounds for $a_d(K_n)$ where $K_n$ is the complete graph on $n$ vertices. In particular, the authors showed that for $n \geq 2d$,
\begin{equation*}
    \left \lceil \frac{n}{2d} \right \rceil -2d + 1 \leq a_d(K_n) \leq \frac{2n}{3(d-1)} + \frac{1}{3},    
\end{equation*}
where $\lceil \cdot \rceil$ denotes the ceiling function.
Also, the authors presented the exact characterization of the stiffness matrix spectrum for the regular $d$-simplex and for a class of balanced Tur\'an graphs.
Although these results are promising, very little is known regarding the generalized algebraic connectivity of arbitrary graphs. Finding bounds for $a_d(G)$ in any dimension without making assumptions about $G$, is a difficult task. However, we believe fundamental relationships between $a_d$ and other better-known graph invariants should be studied. 
In particular, relating $a_d$ with $a_1$ is interesting since the latter is perhaps the most studied algebraic property of a graph. 
The first attempt to do this was also done by \cite[Theorem 4.2]{Jordan2020}. There, it is stated that for every realization $p$,
\begin{equation}
    \lambda_{k}(L_G(p)) \leq \lambda_{\left \lceil \frac{k}{d} \right \rceil}(\mathcal{L}_G), \qquad 1 \leq k \leq d|V|,
    \label{eq:jordan_bound_0}
\end{equation}
which implies that
\begin{equation}
    a_d(G) \leq a_1(G) \quad \text{(for $d=2$)}.
    \label{eq:jordan_bound_2}
\end{equation}
This result extends the the well known notion that a rigid graph must be connected. 
It says that every realization in $\reals[2]$ of a poorly connected graph will render either a flexible or a poorly rigid framework, with the connectivity and rigidity eigenvalues as metrics.
However, \eqref{eq:jordan_bound_0} fails to generalize \eqref{eq:jordan_bound_2} for $d > 2$, since in such case, it gives an upper bound for $a_d$ in terms of an eigenvalue $\lambda_{k}(\mathcal{L}_G)$ where $k > 2$.
This motivated us to dive into the details of \eqref{eq:d_algebraic_connectivity} and \eqref{eq:jordan_bound_0} to find out that \eqref{eq:jordan_bound_2} actually holds for any dimension $d$.

\section{Main Results}
We start by analyzing the implications of \eqref{eq:jordan_bound_0}. Let $k = D + 1$ where $D = (m + 1)(2d - m)/2$ and $m = \dim(\{p_i\}_{i \in V})$. the eigenvalue index in the right-hand side of inequality \eqref{eq:jordan_bound_0} can be written as
\begin{equation*}
    \left \lceil \frac{D + 1}{d}\right \rceil = \left \lceil m + 1 - \frac{\binom{m+1}{2} - 1}{d}\right \rceil,
\end{equation*}
which is non-decreasing on the values of both $m$ and $d$.
Observe that if $m = 1$ then  $\left \lceil \frac{D + 1}{d}\right \rceil = 2$, and if $m \geq 2$ then $\left \lceil \frac{D + 1}{d}\right \rceil \geq \left \lceil 3 - \frac{2}{d} \right \rceil$. In summary, it holds that $\left \lceil \frac{D + 1}{d}\right \rceil \geq 2$ and equality holds if and only if $d \leq 2$ or $m = 1$.
For $d > 2$, $m > 1$, \eqref{eq:jordan_bound_2} does not provide an upper bound for $a_d(G)$ in terms of the algebraic connectivity $a_1(G)$.
The central goal of this paper is to provide such a bound that can be claimed for arbitrary graphs. Before we present our main theorem, we introduce a few properties of interest to the analysis.

In our first lemma we show that if the position set is contained in a one dimensional affine space of $\reals[d]$, then $L_G(p)$ and $\mathcal{L}_G$ are intrinsically related.
\begin{lemma}
    Let $(G, p)$ be a framework such that $p_i - p_j \in \linspan\{x\}$ for all $i, j \in V$ with $\norm{x} = 1$. Then
    \begin{equation}
        L_G(p) = \mathcal{L}_G \kron x x^T.
        \label{eq:lem_1}
    \end{equation}
    \label{lem:1}
\end{lemma}
\begin{proof}
    Consider the $ij$-th block of size $d \times d$ of $L_G(p)$ . For $i \neq j$,
    \begin{equation*}
        B_{ij} =
        \begin{cases}
            - \delta_{ij} \delta_{ij}^T = - x x^T, \quad & ij \in E, \\
            0, \quad & ij \notin E,
        \end{cases}
    \end{equation*}
    and since $B_{ii} = - \sum_{j \neq i} B_{ij} = n_i x x^T$, where $n_i$ is the number of neighbors of $i$, \eqref{eq:lem_1} holds. Note that $B_{ij}$ can still be defined in this way even if $p_i = p_j$.
\end{proof}
As a consequence of this, when the conditions of Lemma \ref{lem:1} are met, then
\begin{equation*}
    \lambda_1(L_G(p)) = \ldots = \lambda_z(L_G(p)) = 0 \quad \text{and} \quad \lambda_{z + i}(L_G(p)) = \lambda_i(\mathcal{L}_G) \text{ for all } i \in V    
\end{equation*}
where $z = (d-1)|V|$. Moreover, $u_{z+i} = v_{i} \kron x$ where $u_k$ and $v_k$ are the $k$-th eigenvectors of $L_G(p)$ and $\mathcal{L}_G$, respectively.

Now, we make a relationship between quadratic forms for the stiffness and laplacian matrices.
\begin{lemma}
    Let $(G, p)$ be a framework.
    Let $X = \linspan \{x\}$ with $\norm{x} = 1$ be a $1$-dimensional subspace in $\reals[d]$,  $v = (v_i)_{i \in V}$ an arbitrary vector in $\reals[|V|]$ and $u = (u_i)_{i \in V} = v \kron x$. Then
    \begin{equation}
        u^T L_G(p) u \leq u^T (\mathcal{L}_G \kron x x^T) u = v^T \mathcal{L}_G v,
    \end{equation}
    and equality holds if and only if $p_i - p_j \in X$ for all $i, j \in V$.
    \label{lem:2}
\end{lemma}
    
\begin{proof}
    Let $\delta_{ij}(p_X) = \pm x$ be the $ij$-th edge-direction vector for the projected configuration $p_X = (x x^T p_i)_{i \in V}$. Since every $u_i \in X$, then
    \begin{align*}
        u^T L_G(p) u &= \sum_{ij \in E} ((u_i - u_j)^T \delta_{ij})^2 
        \leq \sum_{ij \in E} ((u_i - u_j)^T x)^2 \\ 
        &= \sum_{ij \in E} ((u_i - u_j)^T  \delta_{ij}(p_X))^2 \\
        &= u^T L_G(p_X) u = u^T (\mathcal{L}_G \kron x x^T) u = v^T \mathcal{L}_G v, 
    \end{align*}
    which follows from Lemma \ref{lem:1}. Equality holds if and only if $\delta_{ij} = \pm x$.
\end{proof}

Now we are ready to prove the main statement of this paper.
\begin{theorem}
    Let $G$ be a graph, then for every $d \geq 1$
    \begin{equation}
        a_d(G) \leq a_1(G).
        \label{eq:theorem}
    \end{equation}
\end{theorem}
\begin{proof}
    Consider a graph $G$ on $n$ vertices, and an arbitrary $d$-dimensional realization $p$. It is well known that the eigenvalues of $L_G$ are agnostic to rotations of the framework, i.e., $\lambda_{k}(L_G(p)) = \lambda_{k}(L_G(\mathbf{Q} p))$ where $\mathbf{Q} = I_{n} \kron Q$ and $Q \in \mathrm{SO}(d)$.
    Now, let $x = e_1$ the first element of the standard basis in $\reals[d]$ (any unit vector could be chosen) and let $\mathbf{X} \subset \reals[dn]$ be $n$-dimensional subpace obtained by the $n$-th Cartesian power of $X = \linspan\{e_1\}$. Since $L_G$ is symmetric
    \begin{align*}
        \lambda_{D+1}(L_G(p)) &= \min_{u \in T(\mathbf{Q} p)^{\perp}} \frac{u^T L_G(\mathbf{Q}p) u}{u^T u} \\
        &\leq \min_{u \in (\mathbf{X}^{\perp} + T(\mathbf{Q} p))^{\perp}} \frac{u^T L_G(\mathbf{Q} p) u}{u^T u}. \\
        &\leq \min_{u \in (\mathbf{X}^{\perp} + T(\mathbf{Q} p))^{\perp}} \frac{u^T (\mathcal{L}_G \kron e_1 e_1^T) u}{u^T u} \eqqcolon \alpha(Q),
    \end{align*}
    which follows from Lemma \ref{lem:2}.
    Also, it holds that
    \begin{equation}
        \lambda_2(\mathcal{L}_G) = \lambda_{z + 2}(\mathcal{L}_G \kron e_1 e_1^T) =
        \min_{u \in (\mathbf{X}^{\perp} \oplus \linspan\{t_1\})^{\perp}} \frac{u^T (\mathcal{L}_G \kron e_1 e_1^T) u}{u^T u},
        \label{eq:th_proof_2}
    \end{equation}
    where $z = (d-1)n$, $t_1 = 1_{n} \kron e_1$ represents a translational velocity in the $e_1$ direction. It follows since $\mathbf{X}^{\perp} \oplus \linspan\{t_1\}$ is a $z + 1$ dimensional subspace of $\ker(\mathcal{L}_G \kron e_1 e_1^T)$. Note that both $\alpha(Q)$ and $\lambda_2(\mathcal{L}_G)$ are minimums of the same functional with respect to different search spaces.
    
    Now, consider the following generating set for the subspace $T(\mathbf{Q} p)$
    \begin{equation}
        \{1_{n} \kron e_k : 1 \leq k \leq d\} \cup \{\mathbf{A}_{kl} \mathbf{Q} p : 1 \leq k < l \leq d\},
        \label{eq:th_proof_3}
    \end{equation}
    where $\mathbf{A}_{kl} = I_{n} \kron A_{kl}$ and $A_{kl} = E_{kl} - E_{lk}$ with $E_{kl}$ the $kl$-th element of the standard basis for $d \times d$ matrices. 
    Skew-symmetric matrix $\mathbf{A}_{kl}$ produces in every vertex a rotational velocity contained in the $kl$-th coordinate plane. 
    It is easy to see that $t_1 \in T(\mathbf{Q}p)$, therefore
    \begin{equation*}
        (\mathbf{X}^{\perp} + T(\mathbf{Q} p))^{\perp} \subset (\mathbf{X}^{\perp} \oplus \linspan\{t_1\})^{\perp},
    \end{equation*}
    which implies that $\alpha(Q) \geq \lambda_2(\mathcal{L}_G)$ for all rotations $Q$. Nevertheless, in what follows, we show that there always exists a matrix $\hat{Q} \in \mathrm{SO}(d)$ such that $\alpha(\hat{Q}) = \lambda_2(\mathcal{L}_G)$ therefore $\lambda_{D+1}(L_G(\mathbf{\hat{Q}}p)) \leq \lambda_2(\mathcal{L}_G)$, which completes the proof.
    
    To this end, let $\hat{v} = (\hat{v}_i)_{i \in V}$ be the eigenvector of $\mathcal{L}_G$ associated to $\lambda_2(\mathcal{L}_G)$, then $\hat{u} = \hat{v} \kron e_1$ achieves the minimum in \eqref{eq:th_proof_2}. Thus, the goal is to show that $\hat{u}$ (which is independent of $Q$) in fact belongs to 
    $(\mathbf{X}^{\perp} + T(\mathbf{\hat{Q}} p))^{\perp}$ for some rotation $\mathbf{\hat{Q}} = I_{n} \kron \hat{Q}$, which implies that $\alpha(\hat{Q})$ and $\lambda_2(\mathcal{L}_G)$ are equal. In consequence, we must find a rotation matrix $\hat{Q}$ such that that $\hat{u}$ is orthogonal to every element in \eqref{eq:th_proof_3}. Note that for the translational basis, orthogonality is satisfied regardless of $\hat{Q}$, since $1_n \kron e_k \in \mathbf{X}^{\perp}$ for $2 \leq k \leq d$ and $\hat{u} \in \mathbf{X} \cap \linspan\{t_1\}^{\perp}$. For the rotational basis it holds that $\mathbf{A}_{kl} \mathbf{\hat{Q}} p \in \mathbf{X}^{\perp}$ if $k \geq 2$, therefore it is sufficient to find $\hat{Q}$ such that
    \begin{equation}
        \hat{u}^T \mathbf{A}_{1l} \mathbf{\hat{Q}} p = 0 \quad \text{for all } 2 \leq l < d.
        \label{eq:th_proof_4}
    \end{equation}
    Note that $\hat{u} = (\hat{v}_i e_1)_{i \in V}$ and $\mathbf{A}_{1l} \mathbf{\hat{Q}} p = (A_{1l} \hat{Q} p_i)_{i \in V}$, so rewriting \eqref{eq:th_proof_4} as
    \begin{equation*}
        \sum_{i \in V} \hat{v}_i e_1^T A_{1l} \hat{Q} p_i = \sum_{i \in V} \hat{v}_i e_l^T \hat{Q} p_i = 0 \quad \text{for all } 2 \leq l < d,
    \end{equation*}
    it follows that only the $l$-th coordinate of $\hat{Q} p_i$ is relevant to the computation. Then, if $M = (p_i^T)_{i \in V}$ is the $n \times d$ matrix containing the vertex positions in its rows, and $\hat{q}_l$ the $l$-th row of the rotation matrix $\hat{Q}$, it follows that $M \hat{q}_l = (e_l^T \hat{Q} p_i)_{i \in V}$.
    Therefore, \eqref{eq:th_proof_4} is equivalent to
    \begin{equation*}
        \hat{v}^T M \hat{q}_l = 0 \quad \text{for all } 2 \leq l < d.
    \end{equation*}
    Finally, the search is completed by choosing $\hat{q}_1 \coloneqq \frac{M^T \hat{v}}{\norm{M^T \hat{v}}}$ and $\{\hat{q}_l : 2 \leq l < d\}$ accordingly such that $\hat{Q} \in \mathrm{SO}(d)$. If it is the case that $M^T \hat{v} = 0$, then any rotation $\hat{Q}$ is suitable.
\end{proof}

\printbibliography
\end{document}